%% file: Hom_Dirac_final.tex
\documentclass[a4paper,11pt,reqno]{amsart}

\input commands.tex
\usepackage{lmodern}
\usepackage{amssymb,amsmath,amsthm}
\usepackage{mathtools}
\usepackage[left=2.7cm, right=2.7cm, marginparwidth=2cm, textheight = 24cm]{geometry}
\usepackage{xcolor}
\definecolor{darkred}{rgb}{0.65,0.2,0.2}

\usepackage[colorlinks=true,linkcolor=blue,citecolor=darkred]{hyperref}

\usepackage{enumerate}
\usepackage{graphicx}
\usepackage{bbm}

\theoremstyle{remark} 
\newtheorem{example1}[thm]{Example}

\definecolor{darkblue}{rgb}{0.2,0.2,0.6}

\newcommand{\e}{_{\varepsilon}}
\newcommand{\ke}{_{k,\varepsilon}}
\newcommand{\DD}{\mathbf{D}}
\newcommand{\BB}{\mathbf{B}}
\newcommand{\YY}{\mathbf{Y}}

\newcommand{\dist}{\mathrm{dist}\,}
\newcommand{\diam}{\mathrm{diam}\,}

\newcommand{\la}{\langle}
\newcommand{\ra}{\rangle}

\newcommand{\ceq}{\coloneqq}

\begin{document}
	
	\title[Homogenization of the Dirac operator with position-dependent mass]{Homogenization of the Dirac operator  with position-dependent mass}
	
	\author[A.~Khrabustovskyi]{Andrii Khrabustovskyi}
	\address[A.~Khrabustovskyi]{Department of Physics, Faculty of Science, University of
		Hradec Kr\'{a}lov\'{e}, Rokitansk\'eho 62, 50003 Hradec Kr\'alov\'e, Czech Republic} 
	\email{andrii.khrabustovskyi@uhk.cz}
	
	\author[V.~Lotoreichik]{Vladimir Lotoreichik}
	\address[V.~Lotoreichik]{Department of Theoretical Physics, Nuclear Physics Institute, 	Czech Academy of Sciences, Hlavn\'i 130, 25068 \v Re\v z, Czech Republic}
	\email{lotoreichik@ujf.cas.cz}

	
	\begin{abstract}
		We address the homogenization of the two-dimensional Dirac operator with position-dependent mass. The mass is piecewise constant and supported on 
		small pairwise disjoint inclusions evenly distributed along an $\varepsilon$-periodic square lattice.
		Under rather general assumptions on geometry of these inclusions
		we prove that the corresponding family of Dirac operators converges as $\varepsilon\to 0$ in the norm resolvent sense to the Dirac operator with a constant
		effective mass provided the masses in the inclusions are adjusted to the scaling of the geometry. 
		We also estimate the speed of this convergence in terms of the scaling rates.
	\end{abstract}
	
	\maketitle
	\thispagestyle{empty} 
	
	\section{Introduction}
	The Dirac equation describes
	the behaviour of relativistic spin\mbox{-}$\frac12$ quantum particles. This equation plays an important role in the study of two-dimensional structures
	with honeycomb symmetries such as the artificial material graphene; see~\cite{FW12, FW14} and the references therein. The effective Dirac operator for the graphene is massless, but a non-trivial mass term can appear when the graphene sheet is deposited on another crystal
	such as hexagonal  boron-nitride~\cite{JK14, KUM12, SSZ08}. The corresponding mass term can be position-dependent, when the crystal 
	beneath the graphene sheet has  {bumps} or trenches,  {where the distance between graphene and boron-nitride increases}. Dirac operators with position-dependent mass are already considered in mathematical literature in the limit of a large mass supported in the exterior of a fixed open set~\cite{ALMR19, BBZ22, BCLS19, MOP20, SV19}. In this limit, one observes the convergence to the Dirac operator on a domain with infinite mass boundary conditions,  whose spectral properties are studied in many recent contributions (see, e.g.,~\cite{ABLO21, BLRS23, BFSV17, BK22, LO23}). The asymptotic analysis we perform in this paper  {concerns} a significantly different situation, where not only the mass is changing, but also the support of the mass is varying.         
	
	In the present paper, we will consider the two-dimensional Dirac operator  
	\begin{gather*}
		\sfD_\eps   = -\ii  (\s_1\partial_1 + \s_2\partial_2)   + m\e  \s_3 
	\end{gather*}
	acting in $L^2(\dR^2;\dC^2)$ on the domain $H^1(\dR^2;\dC^2)$,
	where $\sigma_j$  are the Pauli matrices, and $m\e$ (``mass'') is a piecewise-constant non-negative function supported on small pairwise disjoint {inclusions} $D\ke$, $k\in\dZ^2$, distributed along an $\eps$-periodic square lattice, i.e.,   each cell of the lattice contains precisely one inclusion  $D\ke$, see Figure~\ref{fig1}  {below}. 
	Hereinafter, we assume that the units are chosen in such a way   that the speed of light   and reduced Planck's constant are  equal to $1$.
	We will address the homogenization of this operator in the limit when the period $\eps$  of the lattice and the outer radii $d\ke$ of these {inclusions} tend to zero
	(with the same or with different rates), and the values $m\ke$ of the mass $m\e$ on $D\ke$ are appropriately scaled.
	
	In short, the main result of the present paper is  as follows: 
	under appropriately chosen masses $m\ke$, the operator
	$\sfD\e$  {converges} as $\eps\to 0$ to the Dirac operator  {$\sfD$} with constant effective mass $m_\star$, {which is expressed in terms of the geometry of the problem.} 
	The convergence is established in the \emph{norm resolvent sense}, and the estimate on the convergence rate is derived. We impose rather general
	assumptions on the shapes of the inclusions $D\ke$: after being upscaled to the domains with unit outer radius, they have to satisfy a uniform with respect to $\eps$ and $k$ trace inequality, while their inner radii and the smallest non-zero Neumann eigenvalues  have to be bounded away from zero uniformly in $\eps$ and $k$.   Examples obeying these assumptions will be discussed.
	The outer radii $d\ke$ of $D\ke$ are allowed to be of a different order with respect to $\eps$; 
	this lack of uniformity in the sizes of inclusions is compensated by the choice of   $ m\ke$.
	The only restriction we  {impose}  reads 
	\begin{equation}\label{eq:assumption0}
		\lim_{\eps\arr0}\frac{\eps^2}{d_\eps}
		\left(\ln\left(\frac{\eps}{d\e}\right)\right)^{1/2} = 0,\quad\text{ where }\,d\e\coloneqq\inf_{k}d\ke,
	\end{equation} 
	that is the outer radii of the inclusions cannot tend to zero too fast.
	
	The counterpart problem for scalar Schr\"odinger operators was studied by Brillard \cite{B88}. In  {that} work, the author considered the operator $-\Delta_\Omega+h\e\chi\e$ in $L^2(\Omega)$, where 
	$\Omega\subset\dR^n$, {$n\ge2$} is a bounded domain, $  {-}\Delta_\Omega$ is the Dirichlet Laplacian on
	$\Omega$, $h\e$ is a positive constant, and $\chi\e$ is the indicator function of the union of $\eps$-periodically distributed  identical domains of the form $D\ke\cong d\e D$,
	where $0<d\e\le C\eps$ and $D\subset\dR^n$ is a fixed  {open} set.
	It was shown for $n\ge 3$ that if
	$\lim_{\eps\to 0}\eps^{n/(n-2)}d\e^{-1}=0$
	and $h_\star\coloneqq\lim_{\eps\to 0} h\e|D\ke|\eps^{-n}>0$, then for each $f\in L^2(\Omega)$ one has
	\begin{gather}\label{wrc}
		(-\Delta_\Omega+ {h_\eps\chi_\eps}) ^{-1}f  \to (-\Delta_\Omega+ {h_\star})^{-1} f\,\text{ weakly in }H^1_{ {0}}(\Omega)\,\text{ as }\eps\to 0
	\end{gather}
	(due to the Rellich theorem, the above convergence holds also in the strong sense in $L^2(\Omega)$).  If $ {d\e \sim \eps^{n/(n-2)}}$, 
	the result remains qualitatively the same: the limiting operator is  a Schr\"odinger operator 
	with a constant   potential, but now this potential equals the minimimum of some 
	capacity-type functional whose form depends on how $h\e$ scales.
	The proofs are based on epi-convergence methods.
	For $n=2$ the convergence result in \cite{B88} is formulated only through certain implicit assumption 
	(no explicit restrictions similar to  {\eqref{eq:assumption0}} were given).
	However, tracing the proofs in \cite{B88}, one might conclude that   \eqref{wrc} holds for $n=2$ provided   
	\begin{gather}\label{n=2}
		\eps^2|\ln d\e|\to 0\quad\text{ as }\quad\eps\to 0.
	\end{gather}
	The   scalings $d\e\sim \eps^{n/(n-2)}$ ($n\ge 3$)  and $|\ln d\e|^{-1}\sim \eps^2$ ($n=2$)
	arise  in homogenization of Dirichlet and Robin Laplacians in domains with  holes  
	\cite{B24,CM97,K89,KP22,KP18,MK64}:
	if the holes are scaled as above, in the limit one arrives  {to} a Schr\"odinger operator with a constant 
	potential known in homogenization community by its nickname ``strange term''.

	As we see, our assumption \eqref{eq:assumption0} is more restrictive than
	\eqref{n=2}.  
	It remains an open question   whether the restriction \eqref{eq:assumption0} is of a technical nature (we will come to it when estimating the $H^1(\dR^2;\dC^2)$-norm by the graph norm associated with $\sfD\e$), or it has a principle significance, so that for smaller $d\e$ one has qualitatively different limiting behavior of the operator $\sfD\e$.
	
	The main difficulty when inspecting homogenization of a Dirac operator is connected with the fact that  it is an operator of the first order and is not semibounded, while the known methods of homogenization theory are designed primarily to deal with even order below bounded differential operators and based on them evolution equations, see, e.g., the monographs \cite{BLP11,CD99,MK06,ZKO94}. For operators not falling into these standard frameworks one needs to
	invent \textit{ad-hoc} methods. As a result, so far there exist only a couple of works devoted to homogenization of Dirac operators. 
	
	In \cite{K13} the author addressed homogenization of the two-dimensional Dirac operator   with a constant mass and a periodic magnetic potential ${\bf A}\e$ being either of the form ${\bf A}\e(x)={\bf A}(x\eps^{-1})$ (non-singular case)
	or ${\bf A}\e(x)=\eps^{-1}{\bf A}(x\eps^{-1})$ (singular case), where ${\bf A}$ is a real periodic  {divergence-free} vector-valued function
	{with zero average over the period cell.} For both cases the effective operators were found, in the non-singular case it coincides with a free Dirac operator. The convergence is established
	in the norm resolvent sense  with order-sharp error estimates.
	The scheme of the proof in \cite{K13} is as follows: first the solution to the underlying resolvent equation  is split into two scalar components, each of which solves the resolvent equation for certain second order operators; then for each component the 
	operator-theoretic approach by Birman and  Suslina \cite{BS04,BS06,S11} is applied.
	Actually, the author used the well-known fact that the square of the two-dimensional magnetic Dirac operator 
	coincides with a matrix electromagnetic Schr\"odinger operator of a special form (the so-called Pauli operator).

	We also mention the article~\cite{JL01} concerning homogenization of non-stationary systems emerging from the one-dimensional non-stationary Dirac equation. Finally, in   \cite{AS12}  $G$-convergence of a three-dimensional Dirac operator with a constant mass and an oscillating potential, being restricted to the spectral subspaces associated with the eigenvalues in the gap of the essential spectrum,  was investigated.

	In the proof of the main theorem we extensively use the abstract convergence result obtained in Section~\ref{sec:abstract} below. Its main idea is to estimate the absolute value of the difference of the bilinear forms of the operators $\sfD_\eps$ and $\sfD$ in terms of graph-norms associated to these operators. Then, when implementing   this abstract result, we combine 
	several functional estimates  
	with a convenient representation of the quadratic form for the square of the Dirac operator with piecewise constant mass~\cite{BCLS19}.  {An additional
		difficulty
		is caused by the fact that this quadratic form
		has boundary terms over $\partial D\ke$ (cf.~the formula \eqref{BCLS} below).}
	Apparently, the proposed method can be extended to homogenization of Dirac operators in three and higher space dimensions.

	\section{Setting and the main result}\label{sec:main}
	
	\subsection{Geometric setting}
	
	Let $\YY\subset\dR^2$ be the open unit square centred at the origin.
	For $k\in\dZ^2$ we set $Y\ke\ceq \eps \YY+ \eps k$.
	Assume that for each $\eps\in (0,\eps_0]$ with  $0<\eps_0<1$ we are given with a family of  {$C^{2,1}$}-smooth (connected) domains
	$\{D\ke,\ k\in \dZ^2\}$  in $\dR^2$  
	satisfying
	\begin{gather}\label{De}
		{D\ke} \subset Y\ke.
	\end{gather}
	Further assumptions on $D\ke$ will be given below, see \eqref{assump:de}--\eqref{assump:N}.
	The above geometrical configuration is presented in Figure~\ref{fig1}.
	\begin{figure}[ht]
		
		\begin{picture}(270,165)
			
			\includegraphics[scale=0.6]{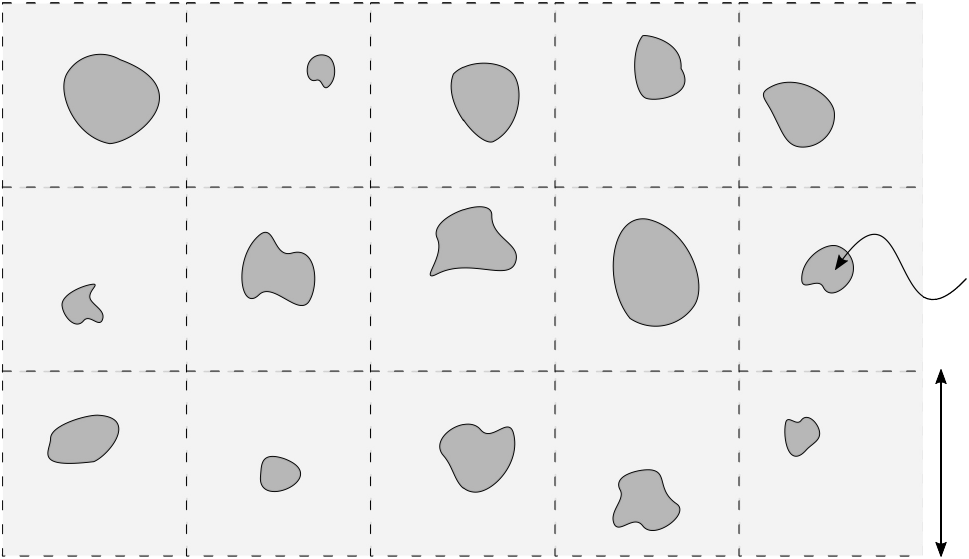} 
			
			\put(-7,25){$\eps$}
			\put(-5,83){$D\ke$}
			\put(-63,60){$Y\ke$}
			
		\end{picture}
		
		\caption{The union of 
			dark grey sets $D\ke$ constitutes the set  supporting the mass.
			\label{fig1}}
	\end{figure}

	Let $d\ke>0$ stand for  the radius of the smallest disk containing $D\ke$. 
	For simple geometric reasons, we have the inequality 
	\begin{gather}\label{22}
		d\ke \le \frac{\eps}{\sqrt{2}}.
	\end{gather}
	The first assumption we impose on $D\ke$ reads as follows:
	\begin{enumerate}[$\bullet$]
		\item For fixed $\eps\in (0,\eps_0]$ the numbers $d\ke$ are bounded 
		away from zero uniformly in $k$, i.e.
		\begin{gather}
			\label{assump:de}
			\forall\, \eps\in (0,\eps_0]\colon\quad    d\e\ceq \inf_{k\in\dZ^2}d\ke>0.
		\end{gather}
	\end{enumerate}
	Further assumptions   will be formulated
	in terms of the
	upscaled sets 
	$\DD\ke\ceq d\ke^{-1}D\ke$.
	Note that the  radius of the smallest disk containing $\DD\ke$ equals $1$.
	In the following, by
	$\Lambda_{\rm N}(\Omega)>0$ we denote  the smallest non-zero eigenvalue of the Neumann Laplacian on
	a  bounded connected open set   $\Omega\subset\dR^2$ with Lipschitz boundary.
	The next assumptions are as follows:
	\begin{enumerate}[$\bullet$]
		\item
		The inner radii of the sets $\DD\ke$ are bounded away from zero uniformly in $\eps$ and $k$, i.e.
		\begin{gather}\label{assump:inrad}
			\exists\,\rho\in (0,1)\ \forall\, \eps\in (0,\eps_0]\ \forall\, k\in\dZ^2\ \exists\, x\ke\in \DD\ke:\quad
			\big\{x\in\dR^2\colon\ |x-x\ke|<\rho  \big\}\subset \DD\ke.
		\end{gather}

		\item The   operators mapping $u\in H^1(\DD\ke)$ to its trace 
		on $\partial\DD\ke$ are bounded in the $H^1\to L^2$ operator norm uniformly in $\eps$ and $k$, i.e.
		\begin{gather}\label{assump:tr}
			\exists\, C_{\rm tr}>0\ \forall\, \eps\in (0,\eps_0]\ \forall\, k\in\dZ^2\ \forall\,
			u\in H^1(\DD\ke)\colon\quad
			\|u\|_{L^2(\partial \DD\ke)}\leq C_{\rm tr}\|u\|_{H^1(\DD\ke)}
		\end{gather}
		(to simplify the presentation, hereinafter we use the same notation for an $H^1$-function on a domain and for its trace on the boundary of this domain).
		
		\item  The  {smallest non-zero} Neumann eigenvalues $\Lambda_{\rm N}(\DD\ke)$ are bounded  away from zero uniformly in $\eps$ and $k$, i.e.
		\begin{gather}\label{assump:N}
			\exists\, \Lambda_{\rm N}>0\ \forall\, \eps\in (0,\eps_0]\ \forall\, k\in\dZ^2\colon\quad
			\Lambda_{\rm N}(\DD\ke)\geq \Lambda_{\rm N}.
		\end{gather}

	\end{enumerate}
	
	Evidently,
	assumptions \eqref{assump:inrad}--\eqref{assump:N} hold
	if 
	the sets $\DD\ke$ are identical (up to a rigid motion).
	While searching for more subtle examples, 
	the main difficulties arise when checking the most implicit 
	assumptions \eqref{assump:tr} and \eqref{assump:N}.
	Analysing the standard proof of the trace inequality
	$\|u\|_{L^2(\partial \Omega)}\leq C\|u\|_{H^1(\Omega)}$ (see, e.g., \cite[\S 5.5, Theorem 1]{E10}), which
	is based on a local straightening of the boundary of a domain,
	one can easily show that the assumption \eqref{assump:tr} holds if, e.g.,
	each $\DD\ke$ possesses a global tubular coordinates for a $\delta$-neighborhood of $\partial \DD\ke$
	(with $\delta$ being independent of $\eps$ and $k$) and the  curvatures of $\partial  {\DD}\ke$ are bounded
	uniformly in $\eps$ and $k$.
	Next, we discuss three examples aiming to realize the assumption
	\eqref{assump:N}.
	
	\begin{example1}
		If  $\Omega\subset\dR^2$ is a  convex domain, then one has the Payne-Weinberger bound \cite{PW60}  $$\Lambda_{\rm N}(\Omega)\geq \frac{\pi^2}{(\diam  \Omega)^2}.$$
		Since the  radius of the smallest disk containing $\DD\ke$ is equal to $1$, we get $\diam \DD\ke\leq 2$. Hence, if the sets $\DD\ke$ are convex, then
		the assumption \eqref{assump:N} holds true with $\Lambda_{\rm N}=\frac{\pi^2}{4}$.
		
	\end{example1}
	
	\begin{example1}
		If $\Omega\subset\dR^2$ is a smooth domain, which is
		{strictly} star-shaped {with} respect to a point $p\in \Omega$, then one has the following inequality by Bramble and Payne \cite{BP62}:
		\begin{gather}\label{BrPa}
			{\Lambda_{\rm N}}(\Omega)\ge \frac{R_{\min}(\Omega) h(\Omega)}
			{(R_{\max}(\Omega))^2\big[(R_{\max}(\Omega))^2+R_{\min}(\Omega)h(\Omega)\big]} .
		\end{gather}
		Here 
		$R_{\min}(\Omega)\ceq\min\limits_{x \in\partial \Omega} {\rm dist}(x,p)$,
		$R_{\max}(\Omega)\ceq\max\limits_{x \in\partial \Omega} {\rm dist}(x,p)$,
		$h(\Omega)\ceq \min\limits_{x\in\partial \Omega}\la  x-p,\nu {(x)}\ra _{\dR^2}$
		with 
		$\nu$ being the outward unit normal to 
		$\partial \Omega$;   the {strict} star-shaped form of $\Omega$ implies $\la  x-p,\nu {(x)}\ra _{\dR^2}>0$. 
		It follows from \eqref{BrPa} that if the upscaled sets $\DD\ke$ are {strictly}   star-shaped with respect to some points $p\ke\in \DD\ke$, moreover, there is $\rho>0$ and $h>0$ (independent of $\eps$ and $k$) such that 
		\begin{gather}\label{C0}
			\big\{x\in\dR^2\colon |x-p\ke|<\rho  \big\}\subset \DD\ke\text{\quad and\quad }
			\min_{x\in\partial  {\DD}\ke}\la  x-p\ke,\nu {(x)}\ra _{\dR^2}\geq h,
		\end{gather}
		then \eqref{assump:N}  holds with 
		$\Lambda_{\rm N}=\frac{\rho h}{32}$;
		to estimate the denominator in \eqref{BrPa} we use  the inequalities 
		$$
		R_{\min}(\DD\ke)\leq R_{\max}(\DD\ke)\leq \diam \DD\ke\leq 2,\qquad
		\left.\la  x-p\ke,\nu {(x)}\ra _{\dR^2}\right|_{x\in\partial \DD\ke}\leq \diam \DD\ke\le 2.
		$$
		The second assumption in~\eqref{C0} means that the distance from the point $p\ke$ to \emph{any}  tangential line to  
		$\partial \DD\ke$  is bounded away from zero uniformly in $\eps$ and $k$.
	\end{example1}
	
	\begin{example1}
		For general domains $\Omega\subset\dR^2$ (i.e, neither convex nor star-shaped) one can use, e.g.,
		the following estimate by Brandolini, Chiacchio, and Trombetti \cite[Theorem 1.1]{BCT15}:
		\begin{gather}
			{\Lambda_{\rm N}}(\Omega)\ge \frac{(K(\Omega))^2\Lambda_{\rm D}(\Omega^\sharp)}{2\pi} ,
		\end{gather}
		where $\Lambda_{\rm D}(\Omega^\sharp)$ is the smallest eigenvalue of the Dirichlet Laplacian 
		on the disk $\Omega^\sharp\subset\dR^2$ satisfying $ {|\Omega^\sharp|=|\Omega|}$  {(hereinafter the notation $|\Omega|$ stands for the area of a  domain $\Omega$)}, and $K(\Omega)  {> 0}$ is
		the best isoperimetric constant relative to $\Omega$
		defined as in~\cite[Eq. (1.3)]{BCT15}. Since the radius of the smallest ball containing
		$\DD\ke$ equals $1$, then the Dirichlet eigenvalues 
		$ {\Lambda_{\rm D}}(\DD\ke^\sharp)$ are bounded from below by the first Dirichlet eigenvalue
		of the unit disk. Therefore the assumption \eqref{assump:N} holds if the best isoperimetric constants $K(\DD\ke)$ are bounded from below uniformly in $\eps$ and $k$.
		
	\end{example1}

	\subsection{Dirac operator with position-dependent mass and the main result}
	
	Let us recall that the Pauli matrices are defined by
	\[
	\s_1 := \begin{pmatrix} 0 & 1\\
		1&0\end{pmatrix},\qquad \s_2 := \begin{pmatrix} 0& -\ii\\
		\ii &0\end{pmatrix},\qquad\s_3 := \begin{pmatrix} 
		1&0\\
		0&-1\end{pmatrix}.
	\]
	For a vector ${x} = (x_1,x_2)\in\dR^2$ we use the the abbreviation
	$\s\cdot{x} = \s_1x_1 + \s_2x_2$. Based on this definition the differential expression $\s\cdot\nb = \s_1\p_1+\s_2\p_2$ is well defined.
	For a fixed $m_\star > 0$, we introduce the numbers
	\begin{gather}\label{me}
		m\ke := m_\star\frac{|Y\ke|}{|D\ke|} =		
		m_\star\frac{\eps^2}{d\ke^2|\DD\ke|},\quad \eps\in (0,\eps_0],\ k\in\dZ^2.
	\end{gather} 
	We study the following $\eps$-dependent family of Dirac operators acting in the space $L^2(\dR^2;\dC^2)$:
	\begin{gather*}
		\sfD_\eps u := -\ii(\s\cdot\nabla) u  + m\e\s_3 u,\qquad \dom\sfD_\eps := H^1(\dR^2;\dC^2).
	\end{gather*}
	Here $m\e$ is the piecewise constant function
	being given by
	$$m\e\ceq\sum_{k\in \dZ^2} m\ke {\chi_{\ke}},
	$$
	where $\chi\ke$ stands for the characteristic function of $D\ke$;
	note that, by virtue of \eqref{assump:de}--\eqref{assump:inrad}, one has $m\e\in L^\infty(\dR^2)$
	(namely, $m\e\leq \frac{m_\star\eps^2}{d\e^2\pi\rho^2}$), whence
	the operator above is well-defined.
	
	The Dirac operator $\sfD_\eps$ is self-adjoint in the Hilbert space $L^2(\dR^2;\dC^2)$ as a bounded perturbation of the self-adjoint free Dirac operator.
	We would like to address the following question.
	
	\smallskip
	
	\begin{center}
		{\it What is the limit of $\sfD_\eps$ as $\eps\arr 0  $ in the norm resolvent sense?}
	\end{center}
	\smallskip
	
	\noindent The candidate for the limit is the following Dirac operator with constant mass
	\[
	\sfD u := -\ii(\sigma\cdot\nabla) u + m_\star\s_3 u,\qquad \dom\sfD := H^1(\dR^2;\dC^2).
	\]
	We prove this convergence under additional assumption on $d_\eps$.
	The main result of the present paper is the following theorem;     note that the logarithm $\ln (\frac{\eps}{d\e})>0$ appearing in this theorem   is positive due to \eqref{22}.
	
	\begin{thm}\label{thm:convergence1} 
		Assume that 
		\begin{equation}\label{eq:assumption}
			\lim_{\eps\arr0}\frac{\eps^2}{d_\eps}
			\left(\ln\left(\frac{\eps}{d\e}\right)\right)^{1/2} = 0.
		\end{equation}
		Then $\sfD_\eps$ converges to $\sfD$ in the norm resolvent sense as $\eps\arr0$ and the following bound on the norm of the difference of their resolvents
		\begin{gather}\label{NRC}
			\big\|(\sfD_\eps - \ii)^{-1} - (\sfD - \ii)^{-1}\big\| \le C {\frac{\eps^2}{d_\eps}
				\left(\ln\left(\frac{\eps}{d\e}\right)\right)^{1/2}}
		\end{gather}
		holds. The constant $C > 0$ in \eqref{NRC}  depends only on $\eps_0$, $m_\star$,  $\rho,\   C_{\rm tr}$, and $\Lambda_{\rm N}$.
	\end{thm}

	Applying the above theorem in the special case  $d\ke=d_\eps  \sim  C\eps^\varkappa$ with ${\varkappa} \in [1,2)$ and $C>0$ (for $\varkappa=1$ this constant has to be sufficiently small  in order to fit \eqref{De}), we arrive at the rate of convergence $\cO(\eps^{ {2-\varkappa}}|\ln\eps|^{1/2})$
	for $\varkappa \in (1,2)$ and the rate of convergence $\cO(\eps)$ for $\varkappa = 1$. 
	Our analysis does not cover the situation $\varkappa\ge 2$.
	In the special case $d\ke=d_\eps   \sim C  \eps^2|\ln \eps|^\omg$ with $\omg >  {\frac12}$ we still get norm resolvent convergence, but the convergence speed is very slow $\cO(|\ln\eps|^{ {1/2-\omg}})$. It remains an open question whether assumption~\eqref{eq:assumption} is necessary for the norm resolvent convergence
	and whether our convergence rate is sharp in $\eps$.
	
	In the above theorem we do not claim $d\ke$ are of the same order. On the contrary, they  can be absolutely different, e.g.
	$d\ke= {\frac12}\eps^{\varkappa_{k,\eps}}$ with some $\varkappa\ke\in [1,\varkappa]$, $1\le\varkappa<2$.
	The main result  is ensured  owing to the  
	choice of  the masses $m\ke$ being adjusted according to \eqref{me}.
	
	Our main result implies, in particular, the  convergence of spectra. Note that 
	the spectrum $\sigma(\sfD)$ of the operator $\sfD$ is purely absolutely continuous and coincides with the set $\dR\setminus (-m_\star,m_\star)$, whence, in particular,  the operator $\sfD$
	is invertible, and 
	\begin{gather}\label{gap}
		\sigma(\sfD^{-1})=[-m_\star^{-1},m_\star^{-1}].
	\end{gather}
	
	Recall that
	for compact sets $X,Y\subset\dC$  the \emph{Hausdorff distance} $\mathrm{dist}_{\rm H} (X,Y)$ is given by
	\begin{gather*}
		\mathrm{dist}_{\rm H} (X,Y)=\max\left\{\sup_{x\in X} \inf_{y\in Y}|x-y|,\,\sup_{y\in Y} \inf_{x\in X}|y-x|\right\}.
	\end{gather*}	
	We also denote by  $\eta\e$  the right-hand side of the estimate   \eqref{NRC}.
	
	\begin{cor}\label{coro}
		Assume that $\eps >0 $ is sufficiently small in order to have 
		\begin{gather}\label{eta:bound}
			\eta\e\leq \frac{1-(1+m_\star^2)^{-1/2}}{1+(1+m_\star^2)^{1/2}}. 
		\end{gather}
		Then   the  operator $\sfD\e$ is invertible with a bounded inverse, and one has the estimate 
		\begin{gather}\label{Haus}
			{\mathrm{dist}_{\rm H}}\big(\sigma( \sfD\e^{-1}),[-m_\star^{-1},m_\star^{-1}]\big)\leq 
			\wt C \frac{\eps^2}{d_\eps}\left(
			\ln\left(\frac{\eps}{d\e}\right)\right)^{1/2},\quad \wt C>0\text{ is a constant}.
		\end{gather}
	\end{cor}

	\begin{proof}
		Using the triangle inequality, the estimate \eqref{NRC} and the fact that for any normal operator $\sfT$ and $\lambda\in\rho(\sfT)$ one has 
		$\|(\sfT-\lambda)^{-1}\|=(\dist(\lambda,\sigma(\sfT)))^{-1}$, we obtain
		\begin{align}\notag
			{ \dist(0,\sigma(\sfD\e))}&\geq    
			{\dist(\ii,\sigma(\sfD\e))-\dist(\ii,0)} =\|(\sfD\e-\ii)^{-1}\|^{-1}-1 
			\geq  \left(\eta\e + \|(\sfD-\ii)^{-1}\| \right)^{-1}-1\\&\notag 
			=
			\left(\eta\e+(\dist(\ii,\sigma(\sfD)))^{-1}\right)^{-1}-1=
			\left(\eta\e+(1+m_\star^2)^{-1/2}\right)^{-1}-1\\&\geq ((1+m_\star^2)^{1/2}-1)/2  \label{dist0}
		\end{align}
		(the penultimate step in \eqref{dist0} relies on $\sigma(\sfD)=\dR\setminus (-m_\star,m_\star)$, while
		the last inequality follows from the assumption \eqref{eta:bound}).
		By \eqref{dist0}, the operator $\sfD\e^{-1}$ is well defined and bounded.
		
		For normal bounded operators $\sfS$, $\sfT$ in the Hilbert space $\mathcal{G}$ one has  {by~\cite[Lemma A.1]{HN99}} the inequality
		$
		{\mathrm{dist}_{\rm H}}(\sigma(\sfS),\sigma(\sfT))\leq\|\sfS-\sfT\|
		$.
		Applying it for $\sfS= \sfD\e ^{-1}$, $\sfT= \sfD ^{-1}$ we get
		\begin{gather}\label{Haus:i}
			{\mathrm{dist}_{\rm H}}\big(\sigma(  \sfD\e ^{-1}),\sigma(  \sfD ^{-1})\big)\leq 
			\| \sfD\e ^{-1}- \sfD ^{-1}\|. 
		\end{gather}
		Then the desired estimate \eqref{Haus} follows from  \eqref{NRC}, \eqref{gap}, \eqref{Haus:i}, the identity
		$$
		\sfD\e ^{-1}- \sfD ^{-1}=-\left(\sfD\e^{-1}+\ii\right) 
		\left((\sfD\e-\ii)^{-1}-(\sfD-\ii)^{-1}\right)
		\left(    \sfD^{-1}+\ii\right) ,
		$$
		and the fact that, by \eqref{dist0}, 
		the norms $\|\sfD\e^{-1}\|$ are uniformly bounded with respect to $\eps$.
	\end{proof}
	
	\begin{remark}
		In the special case when the mass $m_\eps$ is the indicator function of the union
		of $\eps$-periodically distributed identical domains $D_{k,\eps} \cong \eps D$, where $D\subset{\bf Y}$ is a fixed $C^{2,1}$-smooth (connected) open set, we have $d_\eps\sim C\eps$ and we get from Theorem~\ref{thm:convergence1} that
		\[	
		\big\|(\sfD_\eps - \ii)^{-1} - (\sfD-\ii)^{-1}\big\| = \cO(\eps),\qquad \eps\arr0.
		\]
		In this setting, one can also check that $\sfD_\eps^2$ is a special case of the second-order matrix
		elliptic
		operator considered in~\cite{S11}, thanks to 
		very weak regularity assumptions on the coefficients there. It follows then from
		\cite[Theorem 9.2]{S11} that
		\[	
		\big\|(\sfD_\eps^2 + 1)^{-1} - (\sfD^2+1)^{-1}\big\| = \cO(\eps),\qquad \eps\arr0.
		\]
		The general scheme in~\cite{S11} does not apply to the non-periodic case and also to the case when $d_\eps = o(\eps)$ considered in the present paper.
	\end{remark}
	
	The rest of the paper is organized as follows.   
	In Section~\ref{sec:abstract} we will present the above mentioned abstract result.
	In Section~\ref{sec:aux} we will prove several auxiliary functional estimates.
	The proof of the main  result  will be carried out in Section~\ref{sec:proof}.

	
	\section{Abstract scheme}\label{sec:abstract}
	In this section, we obtain an abstract result, which is useful in studying convergence of not necessarily semibounded self-adjoint operators acting in a Hilbert spaces.  
	Other abstract results for studying resolvent convergence of unbounded operators 
	in Hilbert spaces can be found, e.g., in the monographs~\cite{P12, RS80, W00}. Our abstract condition  is close to the approach developed by Post in~\cite{P06,P12}. Similar strategy for the proof of the norm resolvent convergence (applied in a particular setting of Dirac operators) is implicitly used, e.g., in~\cite[ {Proof of Theorem 1.4}]{DM24}.
	However, we are not aware of any source, where the abstract condition for the norm resolvent convergence as in the theorem below is contained, and we believe that this abstract formulation can be useful in other applications. 
	\begin{thm}\label{thm:abstract}
		Let $\sfD$ and $\wt\sfD$ be self-adjoint operators in a Hilbert space $\cH$.
		Assume also that for all $u\in\dom\sfD$ and all $v\in\dom\wt\sfD$
		\begin{equation}\label{eq:assumption2}
			\big|(\sfD u,  v)_\cH - (u,\wt\sfD v)_{\cH}\big|\le 
			c\left(a\|u\|_\cH^2 + \|\sfD u\|^2_\cH\right)^{1/2}
			\left(b\|v\|_{\cH}^2 + \|\wt\sfD v\|^2_{\cH}\right)^{1/2},
		\end{equation}
		with some $a,b,c>0$.
		Then 
		\[
		\big\|(\wt\sfD - \ii)^{-1} - (\sfD-\ii)^{-1}\big\|\le c\sqrt{(a+1)(b+1)}.
		\]
	\end{thm}
	\begin{proof}
		Let us introduce a shorthand notation for the resolvents: $\sfR := (\sfD -\ii)^{-1}$ and $\wt\sfR := (\wt\sfD -\ii)^{-1}$.
		Let $f,g\in\cH$ be arbitrary and define $u := \sfR f$ and $v := \wt\sfR^* g$. We find a convenient representation for the bilinear form of
		the operator $\wt\sfR -\sfR$:
		\[
		\begin{aligned}
			\big((\wt\sfR  - \sfR)f,g\big)_{\cH}
			&= (f,\wt\sfR^* g)_{\cH} - (\sfR f,g)_{\cH}\\
			&= ((\sfD -\ii) u, v)_{\cH} - (u,(\wt\sfD+\ii) v)_{\cH}\\
			&
			= (\sfD u, v)_{\cH} - ( u,\wt\sfD v)_{\cH},\\
		\end{aligned}
		\]
		From this formula we get using~\eqref{eq:assumption2} 	\begin{align}\notag
			\big|\big((\wt\sfR  - \sfR)f,g\big)_{\cH}\big|&\le 
			c\left(a\|\sfR f\|^2_{\cH} + \|\sfD \sfR f\|^2_{\cH}\right)^{1/2} 
			\left(b\|\wt\sfR^* g\|^2_{\cH} + \|\wt\sfD \wt\sfR^* g\|^2_{\cH}\right)^{1/2}\\ 
			&\le c\sqrt{(a+1)(b+1)}\|f\|_\cH\|g\|_{\cH},\label{eq:limit}
		\end{align}
		from which the desired estimate immediately follows.
		In the last step of~\eqref{eq:limit}, we used   the estimates
		$\|(\sfT\pm\ii)^{-1}\|\le 1$ and $\|\sfT(\sfT\pm\ii)^{-1}\|\le 1$ 
		holding for any self-adjoint operator $\sfT$ in a Hilbert space.
		These estimates are direct consequences of the spectral theorem.
	\end{proof}

	\section{Auxiliary estimates}\label{sec:aux}
	
	In this section we collect several functional estimates which will be employed   in the next section for the proof of Theorem~\ref{thm:convergence1}. Actually, we will make use only of  the bounds \eqref{est:LambdaS}, \eqref{eq:Robin_weak}, \eqref{Neumann:Edelta}, the equalities \eqref{minmax:N}, \eqref{minmax:R}, \eqref{muEdelta}
	and Lemmata~\ref{lemma:5}, \ref{lemma:6}, while 
	\eqref{Robin:Edelta}, \eqref{trace:Edelta} and Lemmata~\ref{lemma:1}--\ref{lemma:4} are required only for the proof of Lemmata~\ref{lemma:5}, \ref{lemma:6}.
	\smallskip
	
	\subsection{Estimates for principal Neumann, Robin, and Steklov-type eigenvalues}
	In the first subsection we recall some known properties of the Neumann, Robin and Steklov-type eigenvalues, and also discuss  related inequalities.

	Let  $\Omega\subset\dR^2$ be a  bounded connected open set with Lipschitz boundary.
	Recall that by $\Lambda_{\rm N}(\Omega) > 0$  we  denote the first non-zero eigenvalue of the Neumann Laplacian on $\Omega$.
	This eigenvalue admits the   variational characterisation
	\begin{gather}\label{minmax:N}
		\Lambda_{\rm N}(\Omega) = \inf_{\begin{smallmatrix}f\in H^1(\Omega)\sm\{0\}\\ (f,\one)_{L^2(\Omega)} = 0\end{smallmatrix}}\frac{\|\nabla f\|^2_{L^2(\Omega)}}{\|f\|^2_{L^2(\Omega)}},
	\end{gather}
	where $\one$ denotes the characteristic function of $\Omega$.
	
	Along with $\Lambda_{\rm N}(\Omega)$ we also introduce 
	the number $\Lambda_{\rm S}(\Omega)>0$ via 
	\begin{gather*}
		\Lambda_{\rm S} (\Omega) = \inf_{\begin{smallmatrix}f\in H^1(\Omega)\sm\{0\}\\ (f,\one)_{L^2(\Omega)} = 0\end{smallmatrix}}\frac{\|\nabla f\|^2_{L^2(\Omega)}}{\|f\|^2_{L^2(\partial \Omega)}}.
	\end{gather*}
	It is known that  {$\Lambda_{\rm S}(\Omega) $} is  the smallest non-zero eigenvalue of a Steklov-type spectral problem on $\Omega$ specified, e.g., in~\cite[Section 2]{GS08} in a slightly different notation. {We prefer to use the term `Steklov-type', because this spectral problem differs from the conventional Steklov spectral problem~\cite{GP17}.} From \eqref{minmax:N} and the above definition of $\Lambda_{\rm S}(\Omega)$ one can easily infer the inequality
	\begin{gather}\label{est:LambdaS}
		\Lambda_{\rm S} (\Omega)\geq \frac1{ {(C_{\rm tr}(\Omega))^2} (1+(\Lambda_{\rm N}(\Omega))^{-1})}.
	\end{gather}
	Hereinafter, the notation $C_{\rm tr}(\Omega)$ stands for 
	the norm of the trace operator $H^1(\Omega)\to L^2(\partial \Omega)$.
	
	Finally, for $\gamma\in\dR$ we define the following number $\Lambda_{\rm R}^\gamma(\Omega)$:
	\begin{gather}\label{minmax:R}
		\Lambda_{\rm R}^\gamma(\Omega) = \inf_{f\in H^1(\Omega)\sm\{0\}}\frac{\|\nabla f\|^2_{L^2(\Omega)}+\gamma\|f\|^2_{L^2(\partial \Omega)}}{\|f\|^2_{L^2(\Omega)}},
	\end{gather}
	which is nothing, but   the smallest  eigenvalue of the Laplacian on 
	$\Omega$ subject to the Robin conditions $\frac{\partial  f}{\partial   \nu} + \gamma f=0$ on $\partial \Omega$ with $\nu$ being the  unit normal  pointing outwards of $\Omega$.
	One has the following estimate \cite[Eq.~15]{GS08}
	(keep in mind, that in \cite{GS08} Steklov-type eigenvalues are defined in a different way, below we reformulate their result using our notations):
	\begin{equation}\label{eq:Robin_weak}
		\Lambda_{\rm R}^\gamma(\Omega) \geq   \gamma\frac{|\p \Omega|}{|\Omega|}\left(1-
		\sqrt{-\frac{\gamma}{\Lambda_{\rm S}(\Omega)}}\right)^{-2}\ \text{ provided } -\Lambda_{\rm S}(\Omega)<\gamma<0.
	\end{equation}
	
	We are particularly interested in how the  eigenvalues $\Lambda_{\rm N}(\Omega)$ and $\Lambda_{\rm R}^\gamma(\Omega) $ are influenced under a domain rescaling. 
	Let $\delta>0$ and the set $\Omega_\delta$ be congruent to $\delta \Omega$. 
	It follows easily from \eqref{minmax:N} and \eqref{minmax:R} that   
	\begin{gather}\label{muEdelta}
		\Lambda_{\rm N}( \Omega_\delta ) = \delta^{-2}\Lambda_{\rm N}(\Omega),\quad
		\Lambda_{\rm R}^\gamma(\Omega_\delta) = \delta^{-2}\Lambda_{\rm R}^{\delta\gamma}(\Omega).
	\end{gather}
	Equalities \eqref{muEdelta} imply  the estimates
	\begin{align}\label{Neumann:Edelta}
		\forall\, f\in H^1(\Omega_\delta),\, (f,\one)_{L^2( {\Omega_\delta})}=0\colon&  
		\|f\|_{L^2 (\Omega_\delta)}^2\leq  (\Lambda_{\rm N}(\Omega))^{-1}\delta^2 \|\nabla f\|_{L^2(\Omega_\delta)}^2,
		\\\label{Robin:Edelta}
		\forall f\in H^1(\Omega_\delta)\colon&  
		\|f \|^2_{L^2(\Omega_\delta)}\! \leq\!
		(\Lambda_{\rm R}^1(\Omega))^{-1}\left(\delta\|f \|_{L^2(\partial \Omega_\delta)}^2+\delta^2\|\nabla f\|^2_{L^2(\Omega_\delta)}\right).
	\end{align} 
	
	The last inequality we introduce in this subsection is as follows,
	\begin{align}\label{trace:Edelta}
		\forall f\in H^1(\Omega_\delta):& \quad \|f\|^2_{L^2(\partial \Omega_\delta)}\leq 
		(C_{\rm tr}(\Omega))^2\left(\delta^{-1}\|f\|^2_{L^2(\Omega_\delta)}+\delta\|\nabla f\|^2_{L^2(\Omega_\delta)}\right).
	\end{align}
	The bound \eqref{trace:Edelta} follows easily from
	the   inequality
	$\|f\|^2_{L^2(\partial \Omega)}\leq (C_{\rm tr}(\Omega))^2\big(\|f\|^2_{L^2(\Omega)}+\|\nabla f\|^2_{L^2(\Omega)}\big)$ and the relation $\Omega_\delta\cong\delta \Omega$.

	\subsection{Functional estimates}
	Now, we proceed to the estimates being directly related to our problem.
	Note that all the functions $f$ below  are scalar.
	
	In the following,
	by $\la f\ra _{\Omega}$ we will denote the mean value of a function $f$ in  an open bounded set $\Omega\subset\dR^2$, i.e.
	\begin{equation*}
		\la f\ra_{\Omega}=|\Omega|^{-1}\int_{\Omega}f(x)\dd x,
	\end{equation*}
	where as before $|\Omega|$ stands for the area of $\Omega$. 
	We will keep the same notation for the mean value of a function $f$ on
	a closed curve $S\subset\dR^2$, i.e.
	\begin{equation*}
		\la f\ra_{S}=|S|^{-1}\int_{S}f(x)\dd \s,
	\end{equation*}
	where $\dd \s$ is the density of the surface measure on $S$,
	$|S|=\int_S \dd \s$ stands for the length of $S$.
	
	We also introduce several sets:
	\begin{itemize} \setlength\itemsep{0.4em}
		
		\item $\BB$ is the unit disk  centred at the origin.
		
		\item $B\ke $ is the smallest disk containing  $D\ke$ 
		(i.e. $B\ke$ is congruent to the set $d\ke \BB$).
		
		\item $R\ke$ is the disk of the radius $\eps$ being concentric with $B\ke$.
		
		\item $\wt Y\ke=3\eps \YY + \eps k$, i.e. $\wt Y\ke$
		is obtained from $Y\ke$ by a homothety with the center at $\eps k$ and the ratio $3$.
		
	\end{itemize}
	Taking into account \eqref{22},
	we conclude (cf.~Figure~\ref{fig2}):
	\begin{gather}
		\label{enclo}
		{D\ke}\subset B\ke\subset R\ke\subset \wt Y\ke.
	\end{gather}

	\begin{figure}[ht]
		
		\begin{picture}(225,160)
			
			\includegraphics[scale=0.75]{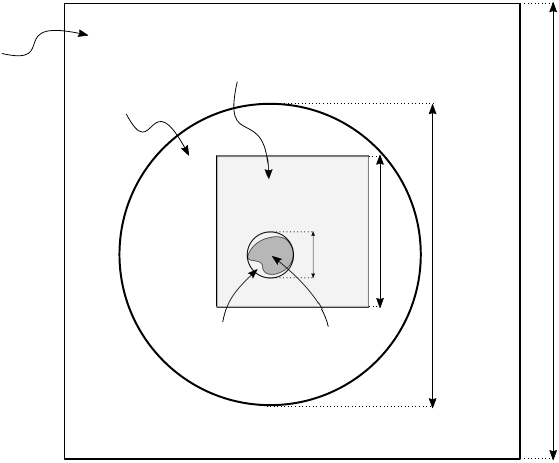} 
			
			\put(-1,82){$_{3\eps}$}
			\put(-64,82){$_\eps$}
			\put(-45,74){$_{2\eps}$}
			\scalebox{0.75}{\put(-118,99){$_{2d\ke}$}}
			
			\put(-124,143){$ {_{Y\ke}}$}
			\put(-212,148){$_{\wt Y\ke}$}
			\put(-163,131){$_{R\ke}$}
			\put(-127,45){$_{B\ke}$}
			\put(-87,44){$_{D\ke}$}
			
		\end{picture}
		
		\caption{The domains being involved in the proof of the main result {of this section.}}\label{fig2}
	\end{figure}
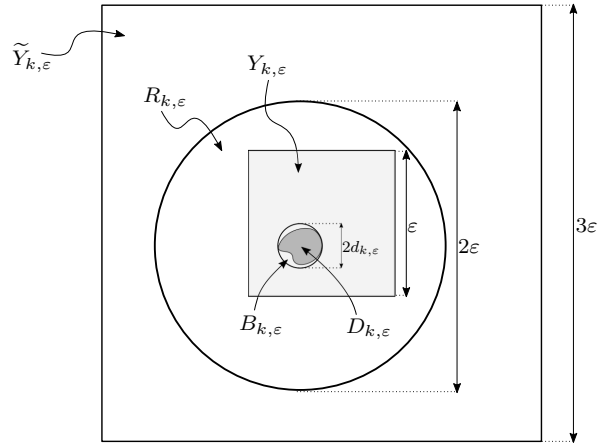

	\begin{lem}\label{lemma:1}
		One has:
		\begin{align}
			\label{lemma:1:est2}
			\forall\, f\in H^1(B\ke)\colon&\quad
			\big|\la  f \ra _{\partial B\ke} - \la  f\ra _{B\ke}\big|^2 \le 
			\frac{(C_{\rm tr}(  \BB))^2((\Lambda_{\rm N}(\BB))^{-1}+1)}{2\pi} \|\nb f\|^2_{L^2(B\ke)}, 
			\\
			\label{lemma:1:est3}
			\forall\, f\in H^1(R\ke)\colon&\quad
			\big|\la  f \ra _{\partial R\ke} - \la  f\ra _{R\ke}\big|^2 \le 
			\frac{(C_{\rm tr}(  \BB))^2((\Lambda_{\rm N}(\BB))^{-1}+1)}{2\pi} \|\nb f\|^2_{L^2(R\ke)}.
		\end{align} 
	\end{lem}
	
	\begin{proof}
		Using \eqref{Neumann:Edelta}, \eqref{trace:Edelta},
		we obtain
		\begin{align}
			\label{Poincare:De}
			\forall f\in H^1(B\ke)\colon& \quad \|f-\la  f\ra _{B\ke}\|_{L^2(B\ke)}^2
			\leq (\Lambda_{\rm N}(\BB)) ^{-1}  d\ke^2\|\nabla f\|_{L^2(B\ke)}^2,
			\\
			\label{trace:De}
			\forall f\in H^1(B\ke)\colon& \quad\|f \|_{L^2(\partial B\ke)}^2\leq 
			(C_{\rm tr}(\BB))^2
			\left(d\ke^{-1}\|  f\|_{L^2(B\ke)}^2+d\ke\|\nabla f\|_{L^2(B\ke)}^2\right).
		\end{align}
		Combining \eqref{trace:De}, \eqref{Poincare:De} and the Cauchy-Schwarz inequality we arrive at the estimate \eqref{lemma:1:est2}:
		\begin{align*} 
			\big|\la  f \ra _{\partial B\ke} - \la  f\ra _{ B\ke}\big|^2&=
			\big|\la  f  - \la  f\ra _{ B\ke}\ra _{\partial B\ke} \big|^2\le
			(2\pi d\ke)^{-1}\|f- \la  f\ra _{ B\ke} \|^2_{L^2(\partial B\ke)}
			\\ 
			&\leq (C_{\rm tr}(\BB))^2(2\pi d\ke)^{-1}
			\left(d\ke^{-1}\|f- \la  f\ra _{B\ke} \|^2_{L^2(B\ke)}+d\ke\|\nabla f \|^2_{L^2(B\ke)}\right)
			\\
			&\leq (C_{\rm tr}(\BB))^2((\Lambda_{\rm N}(\BB))^{-1}+1)(2\pi )^{-1} \|\nabla f \|^2_{L^2(B\ke)}.
		\end{align*}
		The proof of the estimate \eqref{lemma:1:est3} is similar.
	\end{proof}

	\begin{lem}\label{lemma:2}
		One has:
		\begin{equation}\label{lemma:2:est}
			\forall\, f\in H^1(B\ke)\colon\quad
			\big|\la  f \ra _{D\ke} - \la  f\ra _{B\ke}\big|^2 \le 
			\frac{1}{\Lambda_{\rm N}(\BB)\pi\rho^2} \|\nb f\|^2_{L^2(B\ke)}.
		\end{equation}
	\end{lem}
	
	\begin{proof}
		Since $B\ke\cong d\ke \BB$,
		one has by \eqref{Neumann:Edelta}:
		\begin{align}
			\label{Poincare:Be}
			\forall\, f\in H^1(B\ke)\colon & \quad \|f-\la  f\ra _{B\ke}\|^2_{L^2(B\ke)}\leq 
			(\Lambda_{\rm N}(\BB))^{-1}d\ke^2 \|\nabla f\|^2_{L^2(B\ke)}.
		\end{align} 
		Using \eqref{Poincare:Be} and the Cauchy-Schwarz inequality, we obtain
		the desired estimate \eqref{lemma:2:est}:
		\begin{align*}
			\big|\la  f \ra _{D\ke} - \la  f\ra _{B\ke}\big|^2
			&=
			\big|\la  f - \la  f\ra _{B\ke}\ra _{D\ke} \big|^2
			\le
			|D\ke|^{-1}\|f-\la  f\ra _{B\ke}\|^2_{L^2(D\ke)}
			\\
			&\leq 
			|D\ke|^{-1}\|f-\la  f\ra _{B\ke}\|^2_{L^2(B\ke)}
			\le (\Lambda_{\rm N}(\BB))^{-1}|D\ke|^{-1}d\ke^2\|\nabla f\|^2_{L^2(B\ke)}
			\\
			&=
			(\Lambda_{\rm N}(\BB))^{-1}|\DD\ke|^{-1}\|\nabla f\|^2_{L^2(B\ke)}\leq
			( \Lambda_{\rm N}(\BB)\pi\rho^2)^{-1}\|\nabla f\|^2_{L^2(B\ke)},
		\end{align*}
		where in the last step we used \eqref{assump:inrad}.
	\end{proof}

	\begin{lem}\label{lemma:3}
		One has:
		\begin{equation}\label{lemma:3:est}
			\forall\, f\in H^1(R\ke \setminus B\ke)\colon\quad
			\big|\la  f \ra _{\partial R\ke} - \la  f\ra _{\partial B\ke}\big|^2 \le 
			(2\pi)^{-1} {\cdot \ln\left( \tfrac{\eps}{d\ke}\right) }\cdot \|\nb f\|^2_{L^2(R\ke\setminus  {B\ke})}.
		\end{equation}
	\end{lem}
	
	\begin{proof}
		Evidently, it is enough to prove \eqref{lemma:3:est} for $f\in C^1(\overline{R\ke\setminus B\ke})$.
		We introduce the polar coordinate system $(r,\phi)$ with the pole at 
		the center of $B\ke$; here $r>0$ stands for the distance to the pole and 
		$\phi\in [0,2\pi)$ is the angular coordinate. One has
		\begin{align*}
			\la f\ra_{\partial R\ke} - \la f\ra_{\partial B\ke}&=
			(2\pi)^{-1}\left(\int_{0}^{2\pi}f(\eps,\phi)\dd\phi-\int_{0}^{2\pi}f(  d\ke,\phi)\dd\phi\right)\\
			&=(2\pi)^{-1}
			\int_{0}^{2\pi}\int_{d\ke}^{\eps}\frac{\partial  {f}}{\partial r}(\tau,\phi)\dd\tau\dd\phi,
		\end{align*}
		whence, using the Cauchy-Schwarz inequality, we deduce
		\begin{align}\notag
			\big|\la f\ra_{\partial R\ke} - \la f\ra_{\partial B\ke}\big|^2&\leq
			(2\pi)^{-1}
			\int_{0}^{2\pi}\bigg(\int_{d\ke}^{\eps}\left|\frac{\partial  {f}}{\partial r}(\tau,\phi)\right|^2\tau\dd\tau\bigg)
			\cdot \bigg( \int_{d\ke}^{\eps}\tau^{-1}\dd\tau\bigg){\dd \phi}\\
			&\leq
			(2\pi)^{-1}\|\nabla {f}\|^2_{L^2(R\ke\setminus  {B\ke})}
			\ln\left(\tfrac{\eps}{d\ke}\right) .\qedhere \notag
		\end{align}
	\end{proof}
	
	\begin{lem}\label{lemma:4}
		One has:
		\begin{align}\label{lemma:4:est1}
			\forall\, f\in H^1(\wt Y\ke)\colon&\quad 
			\big|\la  f \ra _{R\ke} - \la  f\ra _{\wt Y\ke}\big|^2 \le \frac{9}{\Lambda_{\rm N}(\YY) \pi}\|\nb f\|^2_{L^2(\wt Y\ke)},
			\\
			\label{lemma:4:est2}
			\forall\, f\in H^1(\wt Y\ke)\colon&\quad 
			\big|\la  f \ra _{Y\ke} - \la  f\ra _{\wt Y\ke}\big|^2 \le 
			\frac{9}{\Lambda_{\rm N}(\YY) }\|\nb f\|^2_{L^2(\wt Y\ke)}.
		\end{align}
	\end{lem}
	
	\begin{proof}
		Using \eqref{Neumann:Edelta}, $\wt Y\ke\cong 3\mathbf{Y}$,    
		and the Cauchy-Schwarz inequality, we obtain
		\begin{align*}
			\big|\la  f \ra _{  R\ke} - \la  f\ra _{\wt Y\ke}\big|^2
			&=
			\big|\la  f - \la  f\ra _{\wt Y\ke}\ra _{R\ke} \big|^2
			\le
			|R\ke|^{-1}\|f-\la  f\ra _{\wt Y\ke}\|^2_{L^2(R\ke)}
			\\
			&\leq 
			|R\ke|^{-1}\|f-\la  f\ra _{\wt Y\ke}\|^2_{L^2(\wt Y\ke)}
			\le  9|R\ke|^{-1}(\Lambda_{\rm N}(\YY))^{-1}\eps^2\|\nabla f\|^2_{L^2(\wt Y\ke)}
			\\
			&=
			9\pi^{-1}(\Lambda_{\rm N}(\YY))^{-1}\|\nabla f\|^2_{L^2(\wt Y\ke)}.
		\end{align*} 
		The proof of \eqref{lemma:4:est2} is similar.
	\end{proof}
	
	Now, we are in position to formulate the first key estimate.
	
	\begin{lem}\label{lemma:5}
		One has:
		\begin{equation}\label{lemma:5:est}
			\forall\, f\in H^1(\wt Y\ke)\colon\quad 
			\big|\la  f \ra _{Y\ke} - \la  f\ra _{D\ke}\big|^2 \le 
			C_1  \ln\left( \tfrac{\eps}{d\ke}\right) \|\nb f\|^2_{L^2(\wt Y\ke)},
		\end{equation}
		where the constant $C_1>0$ is given by
		\begin{multline}\label{C1} 
			C_1=\frac{12}{\ln 2}\Bigg(
			\frac{1}{\Lambda_{\rm N}(\BB)\pi\rho^2} +
			\frac{(C_{\rm tr}(\BB))^2((\Lambda_{\rm N}(\BB))^{-1}+1)}{\pi}+\frac{9}{\Lambda_{\rm N}(\YY) \pi}+\frac{9}{\Lambda_{\rm N}(\YY)  }
			\Bigg)+\frac3\pi.
		\end{multline}
	\end{lem}
	
	\begin{proof}
		One has:
		\begin{align}\notag
			\big|\la  f \ra _{Y\ke} - \la  f\ra _{D\ke}\big|^2&\leq
			6\bigg(\big|\la  f \ra _{Y\ke} - \la  f\ra _{\wt Y\ke}\big|^2+
			\big|\la  f \ra _{\wt Y\ke} - \la  f\ra _{R\ke}\big|^2
			\\\notag
			&\qquad+
			\big|\la  f \ra _{R\ke} - \la  f\ra _{\partial R\ke}\big|^2
			+
			\big|\la  f \ra _{\partial R\ke} - \la  f\ra _{\partial B\ke}\big|^2
			\\
			&
			\qquad\qquad+
			\big|\la  f \ra _{\partial B\ke} - \la  f\ra _{B\ke}\big|^2+
			\big|\la  f \ra _{B\ke} - \la  f\ra _{D\ke}\big|^2\bigg)\label{6terms}
		\end{align}
		Combining Lemmata~\ref{lemma:1}--\ref{lemma:4} 
		and taking into account \eqref{enclo} and 
		$\ln(\eps/d\ke) \ge \frac12{\ln 2}$ (cf.~\eqref{22}),
		we infer from \eqref{6terms} the desired estimate \eqref{lemma:5:est}.
	\end{proof}
	
	The second key estimate is given below.
	
	\begin{lem}\label{lemma:6}
		One has
		\begin{equation}\label{lemma:6:est}
			\forall\, f\!\in\! H^1(R\ke)\colon\	\|f\|_{L^2(D\ke)}^2 \le C_2\left\{\left(\frac{d\ke }{\eps }\right)^2\|f\|^2_{L^2(R\ke)}\! +\! d\ke^2
			\ln\left(\tfrac{\eps}{d\ke}\right) \|\nb f\|^2_{L^2(R\ke)}\right\},
		\end{equation}
		where 
		\begin{gather}\label{C2}
			C_2=\frac{\max\big\{2(C_{\rm tr}(\BB))^2,\,2+ 2(\ln 2)^{-1}(1
				+2(C_{\rm tr}(\BB))^2 )\big\}}{\Lambda^1_{\rm R}(\BB)}.
		\end{gather}
	\end{lem}
	
	\begin{proof}
		Using the estimate \eqref{Robin:Edelta} and the inclusion
		{$D_{k,\eps}\subset B_{k,\eps}$} we get:
		\begin{gather}\label{Robin:De}
			\forall\, f\in H^1(B\ke)\colon\quad 
			\|f\|^2_{L^2(D\ke)}\leq (\Lambda^1_{\rm R}(\BB))^{-1}
			\left(d\ke\|f\|^2_{L^2(\partial B\ke)}+d\ke^2\|\nabla f\|^2_{L^2(B\ke)}\right).
		\end{gather}
		
		Next, we prove the inequality
		\begin{equation}\hspace{-2mm}
			\label{DR:est}
			\forall f\!\in\! H^1(R\ke\!\setminus\!\, {B\ke})\colon\, \|f\|_{L^2(\partial B\ke)}^2\!\le \!
			2\left\{\frac{d\ke}{\eps }\|f\|^2_{L^2(\partial R\ke)} \!+\!d\ke \ln\left(\tfrac{\eps}{d\ke}\right) \|\nb f\|^2_{L^2(R\ke\!\setminus {B\ke})}\right\}\!.
		\end{equation}
		Evidently, it is enough to demonstrate \eqref{DR:est} for $f\in C^1(\overline{R\ke {\setminus B_{k,\eps}}})$. 
		We introduce the polar coordinate system $(r,\phi)\in [0,\infty)\times [0,2\pi)$ with the pole at 
		the center of $B\ke$.
		One has
		\begin{align}\notag
			\big|f(d\ke,\phi)\big|^2& = \left|f(\eps,\phi)-
			\int_{d\ke}^{\eps}\frac{\partial f}{\partial r}(\tau,\phi)\dd\tau\right|^2
			\\\notag
			&\leq 2\big|f(\eps,\phi)\big|^2+
			2\bigg(\int_{ d\ke}^{\eps}
			\left|\frac{\partial  {f}}{\partial r}(\tau,\phi)\right|^2\tau\dd\tau\bigg)\cdot \bigg(  \int_{d\ke}^{\eps}\tau^{-1}\dd\tau\bigg)
			\\\label{FTC}
			&\leq 2\big|f(\eps,\phi)\big|^2+
			2\,{ \ln\left(\tfrac{\eps}{d\ke}\right) }\int_{d\ke}^{\eps}\left|\nabla  {f}(\tau,\phi)\right|^2 {\tau}\dd\tau  .
		\end{align}
		Integrating \eqref{FTC} over $\phi$ and then multiplying by $ d\ke$, we arrive at \eqref{DR:est}.
		
		Finally,  using the estimate \eqref{trace:Edelta},
		we obtain
		\begin{gather}\label{trace:Re}
			\forall\, f\in H^1(R\ke)\colon\quad 
			\|f\|^2_{L^2(\partial R\ke)}\leq (C_{\rm tr}( {\BB}))^2 \left(\eps^{-1}\|f\|^2_{L^2( R\ke)}+\eps\|\nabla f\|^2_{L^2(R\ke)}\right).
		\end{gather} 
		The desired estimate \eqref{lemma:6:est} follows from \eqref{Robin:De},  \eqref{DR:est}, \eqref{trace:Re} and   $\ln(\eps/d\ke) \ge \frac12\ln 2$.
	\end{proof}

	\section{Proof of Theorem~\ref{thm:convergence1}}
	\label{sec:proof}
	
	In the following, we will use for an open set $\Omega\subset\dR^2$ the notation
	$\|\cdot\|_{L^2(\Omega)}$ for the norms in $L^2(\Omega;\dC)$, $L^2(\Omega;\dC^2)$ and $L^2(\Omega;\dC^{2\times 2})$ as no confusion can arise. The respective inner products are linear in the first entry and will be denoted by $(\cdot,\cdot)_{L^2(\Omega)}$. Thus, if $u=(u_1,u_2)\in L^2(\Omega;\dC^2)$ (that is $\nabla u\in L^2(\Omega;\dC^{2\times 2})$), then we have 
	\begin{align*}
		\|u\|_{L^2(\Omega)}^2&=\|u_1\|_{L^2(\Omega)}^2+\|u_2\|_{L^2(\Omega)}^2,\\
		\|\nabla  u\|_{L^2(\Omega)}^2&=\|\nabla u_1\|_{L^2(\Omega)}^2+\|\nabla u_2\|_{L^2(\Omega)}^2\\
		&=
		\|{\p_1} u_1\|_{L^2(\Omega)}^2+\|{\p_2} u_1\|_{L^2(\Omega)}^2+
		\|{\p_1} u_2\|_{L^2(\Omega)}^2+\|{\p_2} u_2\|_{L^2(\Omega)}^2.
	\end{align*}
	The same convention will be applied to   $H^1(\Omega;\dC)$ and $H^1(\Omega;\dC^2)$. The norm in the Sobolev spaces $H^1(\Omg;\dC)$ and $H^1(\Omg;\dC^2)$  is defined via the identity $\|u\|^2_{H^1(\Omg)} := \|\nb u\|^2_{L^2(\Omg)} + \|u\|^2_{L^2(\Omg)}$.\smallskip
	
	By the abstract scheme provided in Section~\ref{sec:abstract} it suffices to obtain a suitable upper bound on the absolute value of
	\[
	\frs_\eps[u,v] := (\sfD u,v)_{L^2(\dR^2)} - ( u,\sfD_\eps v)_{L^2(\dR^2)},
	\]
	valid for any $u,v\in H^1(\dR^2;\dC^2)$.
	Using the equality 
		$(\ii(\s\cdot\nabla) u,v)_{L^2(\dR^2)}=(u,\ii(\s\cdot\nabla)v)_{L^2(\dR^2)}$ and the definition of the matrix $\sigma_3$
	we can express $\frs_\eps[u,v]$ as
	\begin{align*}
		\frs_\eps[u,v] &\!=\! \sum_{k\in\dZ^2}\left(
		m_\star\int_{Y\ke} \la \s_3 u(x),v(x)\ra _{\dC^2}\dd x-
		\int_{D\ke}m\ke  \la  u(x),\s_3v(x)\ra _{\dC^2}\dd x\right)\\
		&\!=\!\sum_{k\in\dZ^2}\left(
			m_\star\int_{Y\ke}\!\! \big(u_1(x) \overline{v_1(x)} - u_2(x)  \overline{v_2(x)}\big)\dd x-
			\int_{D\ke}\!\! m\ke  \big(u_1(x) \overline{v_1(x)} - u_2(x)  \overline{v_2(x)}\big)\dd x\right)\!,
	\end{align*}
	whence, we get by triangle inequality
	\begin{equation}\label{eq:R}
		|\frs_\eps[u,v]| \le
		\sum_{j=1}^2 
		\left|m_\star\sum_{k\in\dZ^2}\int_{Y\ke} u_j(x)\ov{v_j(x)}\dd x-
		\sum_{k\in\dZ^2} \int_{D\ke}m\ke  u_j(x)\ov{v_j(x)}\dd x\right|,
	\end{equation}
	where we use the convention that $u = (u_1,u_2)$ and $v = (v_1,v_2)$.
	We will split the rest of the proof into two steps. 
	
	\subsection*{Step 1: Estimates in terms of $H^1$-norms}
	
	Let $f,g\in H^1(\dR^2)$ be arbitrary scalar functions.
	For any $k\in\dZ^2$, 
	there holds
	\begin{align}\notag
		&\left|\int_{D\ke}  f(x)\ov{g(x)}\dd x - \la  f\ra _{D\ke}\la  \ov{g}\ra _{D\ke}|D\ke|\right|\\\notag
		&\qquad=\left|\int_{D\ke}  \left(f(x)-\la  f\ra _{D\ke}\right)\left(\ov{g(x)}-\la \ov{g}\ra _{D\ke}\right)\dd x\right|\\
		&\qquad\le \| f - \la  f\ra _{D\ke}\|_{L^2(D\ke)}\|g-\la  g\ra _{D\ke}\|_{L^2(D\ke)}
		\leq \Lambda_{\rm N}^{-1} d\ke^2 \|\nb f\|_{L^2(D\ke)}\|\nb g\|_{L^2(D\ke)},\label{eq:bndD}
	\end{align}
	where in the last estimate we used the bound \eqref{Neumann:Edelta} and the assumption \eqref{assump:N}.
	Analogously, we obtain for any $k\in\dZ^2$ that
	\begin{equation}\label{eq:bndY}
		\left|\int_{Y\ke}  f(x)\ov{g(x)}\dd x - \la  f\ra _{Y\ke}\la  \ov{g}\ra _{Y\ke}|Y\ke|\right| \le
		(\Lambda_{\rm N}(\YY))^{-1}\eps^2\|\nb f\|_{L^2(Y\ke)}\|\nb g\|_{L^2(Y\ke)}.
	\end{equation}
	Now, using    \eqref{eq:bndD}, \eqref{eq:bndY},  the Cauchy-Schwarz inequality in $\ell^2(\dZ^2)$ and 
	\begin{gather}\label{mede}
		m\ke d\ke^2=\frac{m_\star\eps^2}{ |\DD\ke|}\leq \frac{m_\star\eps^2}{\pi\rho^2}
	\end{gather}
	(the above inequality follows from \eqref{assump:inrad}),
	we derive from~\eqref{eq:R}:
	\begin{align}\notag
		\big|\frs_\eps[u,v]\big| &\le \sum_{j=1}^2\left|\sum_{k\in\dZ^2}\left(m\ke\la  u_j\ra _{D\ke}\la \ov{v_j}\ra _{D\ke}|D\ke| -m_\star\la  u_j\ra _{Y\ke}\la \ov{v_j}\ra _{Y\ke}|Y\ke|\right)\right|\\
		\label{eq:R2}
		&\qquad +
		\eps^2 m_\star \left(\frac{1}{\Lambda_{\rm N}\pi\rho^2}+\frac{1}{\Lambda_{\rm N}(\YY)}\right)\sum_{j=1}^2 \|\nb u_j\|_{L^2(\dR^2)}\|\nb v_j\|_{L^2(\dR^2)}.  	
	\end{align}
	
	Our next aim is to estimate the expression 
	$$\sum_{k\in\dZ^2}\left(m\ke\la  f\ra _{D\ke}\la \ov{g}\ra _{D\ke}|D\ke| -m_\star\la  f\ra _{Y\ke}\la \ov{g}\ra _{Y\ke}|Y\ke|\right),$$
	where $f,g\in {H^1}(\dR^2)$ are arbitrary scalar functions.
	One has:
	\[
	\begin{aligned}
		&\left|\sum_{k\in\dZ^2}\left(m\ke\la  f\ra _{D\ke}\la \ov{g}\ra _{D\ke}|D\ke| -m_\star\la  f\ra _{Y\ke}\la \ov{g}\ra _{Y\ke}|Y\ke|\right)\right| \\
		&\qquad = \bigg|
		\sum_{k\in\dZ^2}\bigg(
		m\ke\la  f\ra _{Y\ke}\la \ov{g}\ra _{Y\ke}|D\ke|-m\ke\la  f\ra _{Y\ke}\la \ov{g}\ra _{Y\ke}|D\ke|\\
		&\qquad\qquad\qquad +	
		m\ke\la  f\ra _{D\ke}\la \ov{g}\ra _{Y\ke}|D\ke|-m\ke\la  f\ra _{D\ke}\la \ov{g}\ra _{Y\ke}|D\ke|\\
		&\qquad\qquad\qquad\qquad+
		m\ke\la  f\ra _{D\ke}\la \ov{g}\ra _{D\ke}|D\ke| -m_\star\la  f\ra _{Y\ke}\la \ov{g}\ra _{Y\ke}|Y\ke|\bigg)
		\bigg|\\
		&\qquad\le \underbrace{\left|m_\star\eps^2\sum_{k\in\dZ^2}
			\la  f\ra _{D\ke}\big(\la  \ov{g}\ra _{D\ke} - \la  \ov{g}\ra _{Y\ke}   \big)\right|}_{=:\sfQ_1^\eps}
		+\underbrace{\left| m_\star\eps^2 \sum_{k\in\dZ^2} 
			\big(\la  f\ra _{D\ke} - \la  f\ra _{Y\ke}\big)\la  \ov{g}\ra _{Y\ke}\right|}_{=:\sfQ_2^\eps},\\
	\end{aligned}	
	\]
	where we used in the derivation of the above bound that $m\ke|D\ke| = m_\star|Y\ke|=m_\star\eps^2$, thanks to which, in particular, the first and the last terms cancelled in the second step of the computation.
	We estimate the term $\sfQ_1^\eps$ by   the Cauchy-Schwarz inequality:
	\[
	|\sfQ_1^\eps| \le m_\star\eps^2\left(\sum_{k\in\dZ^2}\big|\la  \ov{g}\ra _{Y\ke} - \la \ov{g}\ra _{D\ke}
	\big|^2\right)^{1/2}\cdot\left(\sum_{k\in\dZ^2}|\la  f\ra _{D\ke}|^2\right)^{1/2}.
	\]
	By Lemma~\ref{lemma:5} we get that for any $k\in\dZ^2$
	\begin{equation}\label{lemma:4:est}
		\sum_{k\in\dZ^2}\big|\la  \ov{g}\ra _{Y\ke} - \la \ov{g}\ra _{D\ke}\big|^2 \le C_1\ln\left(\tfrac{\eps}{d\e}\right)\sum_{k\in\dZ^2}\|\nb g\|^2_{L^2(\wt Y\ke)}
	\end{equation}
	with $C_1>0$ defined by \eqref{C1}.
	Using Lemma~\ref{lemma:6}, assumption \eqref{assump:inrad} and the Cauchy-Schwarz inequality we get the following bound
	\begin{align}\notag
		\sum_{k\in\dZ^2}|\la  f\ra _{D\ke}|^2& \le |D\ke|^{-1}\|f\|^2_{L^2(D\ke)}\\
		\notag&\le 
		\sum_{k\in\dZ^2}\frac{C_2}{|\DD\ke|}
		\left\{
		\eps^{-2}\|f\|^2_{L^2(R\ke)} +   
		\ln\left(\tfrac{\eps}{d\ke}\right) \|\nb f\|^2_{L^2(R\ke)}	\right\}
		\\\label{eq:bnd1}
		&\le 
		\frac{C_2}{\pi \rho^2}\sum_{k\in\dZ^2}
		\left\{
		\eps^{-2} \|f\|^2_{L^2(R\ke)} + 
		\ln\left(\tfrac{\eps}{d\e}\right)  \|\nb f\|^2_{L^2(R\ke)}	\right\},
	\end{align}
	where $C_2>0$ is given in \eqref{C2}.
	Finally, we observe that   
	\begin{gather}
		\label{YY}
		\forall\, h\in L^2(\dR^2)\colon\quad \sum_{k\in \dZ^2}\|h\|^2_{L^2(R\ke)}\le
		\sum_{k\in \dZ^2}\|h\|^2_{L^2(\wt Y\ke)} = 9\sum_{k\in \dZ^2}\|h\|^2_{L^2(  Y\ke)} =9\|h\|^2_{L^2(\dR^2)} .
	\end{gather}
	It follows easily from \eqref{eq:assumption} that $\lim_{\eps\to 0}\eps^2\ln(\eps/d\e)=0$.
	Assuming further  that $\eps$ is sufficiently small in order to have 
	$\ln(\eps/d\e)\leq \eps^{-2}$,
	we conclude from~\eqref{lemma:4:est}--\eqref{YY} that
	\begin{equation}\label{eq:Q1}
		|\sfQ_1^\eps| \le 9m_\star\sqrt{\frac{C_1C_2}{\pi\rho^2}}
		\eps \left(\ln\left(\tfrac{\eps}{d\e}\right)\right)^{1/2}\|\nb g\|_{L^2(\dR^2)}\|f\|_{H^1(\dR^2)}.
	\end{equation}
	Analogously, we find that
	\begin{equation}\label{eq:Q2}
		|\sfQ_2^\eps| \le 3m_\star\sqrt{ {C_1} }\eps
		\left(\ln\left(\tfrac{\eps}{d_\eps}\right)\right)^{1/2}\|\nb f\|_{L^2(\dR^2)}\|g\|_{L^2(\dR^2)}.
	\end{equation}
	
	As a consequence of the estimates~\eqref{eq:R2},~\eqref{eq:Q1}, \eqref{eq:Q2} 
	and $\ln(\eps/d\ke) \ge \frac12{\ln 2}$ (cf.~\eqref{22}),
	we end up with
	the bound
	\begin{equation}\label{eq:R3}
		\big|\frs_\eps[u,v]\big| \le C_3\eps
		\left(\ln\left(\tfrac{\eps}{d_\eps}\right)\right)^{1/2}\|u\|_{ H^1(\dR^2 )}\|v\|_{ H^1(\dR^2 )},\quad \forall\, u,v\in H^1(\dR^2;\dC^2),
	\end{equation}
	with the constant $C_3> 0$ being given by
	\begin{gather}\label{C3}
		C_3=
		{18}m_\star\sqrt{\frac{C_1C_2}{\pi\rho^2}}+
		{6}m_\star\sqrt{ {C_1} }+
		{2}(\tfrac12\ln 2)^{-1/2} m_\star {\left(\frac{1}{\Lambda_{\rm N}\pi\rho^2}+\frac{1}{\Lambda_{\rm N}(\YY)}\right)},
	\end{gather}
	where $C_1,\,C_2$ are defined by \eqref{C1}, \eqref{C2}.
	\medskip
	
	\subsection*{Step 2: Estimate  in terms of graph norms}
	To apply Theorem~\ref{thm:abstract} we need to bound $H^1$-norms of $u$ and $v$ in \eqref{eq:R3} in terms of graph norms associated with operators $\sfD$ and $\sfD_\eps$.
	
	For any $u\in H^1(\dR^2;\dC^2)$ one can easily check
	via integration by parts that
	$$\|\sfD u\|^2_{L^2(\dR^2)} = \|\nb u\|^2_{L^2(\dR^2)} + m^2_\star\|
	u\|^2_{L^2(\dR^2)},$$
	whence 
	\begin{gather}\label{D:final}
		\|\sfD u\|^2_{L^2(\dR^2)} + \|u\|^2_{L^2(\dR^2)} \ge \|u\|^2_{H^1(\dR^2)}.
	\end{gather}

	The corresponding graph norm associated with the operator $\sfD_\eps$ is more subtle.
	We denote by $\nu_{D\ke}(x)$  the outer unit normal to $D\ke$ at the point $x$.
	For any $x\in \p D\ke$ we define  the mapping
	$\sfB\ke(x) := -\ii\s_3(\s\cdot\nu_{D\ke}(x))\colon\dC^2\arr\dC^2$, which is  self-adjoint  with eigenvalues $\pm 1$, and  the corresponding eigenprojections 
	$\sfP\ke^\pm(x) := \frac{1\pm \sfB\ke(x)}{2}$.
	One has the following equality for $v=(v_1,v_2)\in H^1(\dR^2;\dC^2)$:
	\begin{align}\notag
		\|\sfD_\eps v\|^2_{L^2(\dR^2)} 
		&=
		\|\nb v\|^2_{L^2(\dR^2)} + 
		\sum_{k\in\dZ^2}\int_{D\ke}m\ke^2|v|^2\dd x 
		\\
		&+ 
		\sum_{k\in\dZ^2}\int_{\p D\ke} {m\ke}|\sfP^+\ke v|^2\dd \s\, -
		\sum_{k\in\dZ^2}\int_{\p D\ke} {m\ke} |\sfP^-\ke v|^2\dd \s.\label{BCLS}
	\end{align}
	For $m\ke$ being the same for all $k$ the proof of \eqref{BCLS} can be found in \cite[p.~1885]{BCLS19} {(see also~\cite[Lemma 2]{SV19})}, for the case of different $m\ke$ the proof is absolutely the same.
	We also remark that by careful inspection of~\cite{BCLS19} one sees that this formula is still valid	if the boundaries of $D_{k,\eps}$ and $D_{k',\eps}$	for some $k,k'\in\dZ^2$, $k\ne k'$
	have a non-empty intersection.

	From~\eqref{BCLS} we derive easily
	the following estimate:
	\begin{align}
		\|\sfD_\eps v\|^2_{L^2(\dR^2)} 
		\!\ge\! \frac12\|\nb v\|^2_{L^2(\dR^2)} + 
		\frac12\sum_{k\in\dZ^2}
		\left[\|\nb v\|^2_{L^2(D\ke)} + 2m\ke\int_{\p D\ke}\left(|\sfP\ke^+v|^2-|\sfP\ke^-v|^2\right)\!\dd \sigma\right]\!.\label{eq:lowerbnd}
	\end{align}
	To proceed further we observe that due to \eqref{assump:tr}, \eqref{assump:N}, \eqref{est:LambdaS}, we get
	\begin{gather}\label{est:LambdaS+}
		\forall \eps\in (0,\eps_0]\ \forall k\in\dZ^2\colon\quad\frac{1}{\Lambda_{\rm S}(\DD\ke)} \leq 
		{(C_{\rm tr})^2} (1+\Lambda_{\rm N} ^{-1}).
	\end{gather}
	Furthermore, the assumption \eqref{eq:assumption} being combined with \eqref{assump:inrad} 
	implies 
	\begin{gather}\label{md0}
		{\sup_{k\in\dZ^2}}m\ke d\ke\to 0~\text{ as }~\eps\to 0.
	\end{gather}
	Finally, using the assumption \eqref{assump:tr} (with $u\equiv 1$) and taking into account that
	each $\DD\ke$ is contained in a unit disk, we get
	\begin{gather}\label{per:est}
		|\partial \DD\ke|\leq  {(C_{\tr })^2}|\DD\ke|.
	\end{gather}	
	Notice that for any fixed vector $\xi = (\xi_1,\xi_2)\in\dC^2$ and any $k\in\dZ^2$
	\begin{equation}\label{eq:xi}
		\begin{aligned}
			&\int_{\p D\ke}|\sfP\ke^+\xi|^2\dd \s 
			-		\int_{\p D\ke}|\sfP\ke^-\xi|^2\dd \s \\
			&\qquad\qquad=  2\Re\int_{\p D\ke}\big(-\ii\s_3(\s\cdot\nu_{D\ke}(x))\xi\big)
			\cdot\ov{\xi}\dd \s(x)  = 0,
		\end{aligned}	
	\end{equation}
	where we used that $\int_{\p D\ke}\nu_{D\ke}(x)\dd\s(x) = 0$ in the last step.  
	Taking into account that $|\sfP^\pm\ke v|\le |v|$,
	we deduce from~\eqref{per:est} and~\eqref{eq:xi} using the inequality
	$2xy \le t^{-1} x^2+ ty^2$ ($x,y,t > 0$)
	that for any $v\in H^1(\dR^2;\dC^2)$ and any fixed constant $\aa > 0$ and all $k\in\dZ^2$
	\begin{align}
		\notag&m\ke\int_{\p D\ke}\left(|\sfP\ke^+v|^2-|\sfP\ke^-v|^2\right)\dd \sigma \\
		\notag&\qquad= 
		m\ke\int_{\p D\ke}\left(|\sfP\ke^+(v-\langle v\rangle_{D\ke} + \langle v\rangle_{D\ke})|^2-|\sfP\ke^-(v-\langle v\rangle_{D\ke} + \langle v\rangle_{D\ke})|^2\right)\dd \sigma\\
		\notag&\qquad\ge-m\ke\int_{\p D\ke}|v-\langle v\rangle_{D\ke}|^2\dd\s
		-4m\ke\int_{\p D\ke}|v-\langle v\rangle_{D\ke}|\cdot|\langle v\rangle_{D\ke}|\dd\s\\
		\notag&\qquad\ge -\left(m\ke +\frac{1}{\aa d\ke}\right)\int_{\p D\ke}|v-\langle v\rangle_{D\ke}|^2\dd\s
		- 4m\ke^2 d\ke \aa|\p D\ke|\cdot |\langle v\rangle_{ D\ke}|^2\\
		\notag&\qquad\ge -\left(m\ke +\frac{1}{\aa d\ke}\right)\int_{\p D\ke}|v-\langle v\rangle_{D\ke}|^2\dd\s
		- 4m\ke^2\aa\frac{|\p {\bf D}\ke|}{|{\bf D}_{\ke}|} \|v\|_{L^2(D\ke)}^2\\
		&\qquad	
		\ge -\left(m\ke +\frac{1}{\aa d\ke}\right)\int_{\p D\ke}|v-\langle v\rangle_{D\ke}|^2\dd\s
		- 4m\ke^2\aa(C_{\rm tr})^2 \|v\|_{L^2(D\ke)}^2,
		\label{eq:bterm}
	\end{align}
	where for $v=(v_1,v_2):\dR^2\to\dC^2$ we denote $\langle v\rangle_{ D\ke}\coloneqq (\langle v_1\rangle_{ D\ke},\langle v_2\rangle_{ D\ke})\in\dC^2$.
	In view of~\eqref{md0}, by choosing the value of the constant $\aa > 0$ sufficiently large we can fulfil the conditions
	\begin{equation}\label{eq:conditiona}
		\begin{aligned}	
			\sup_{k\in\dZ^2}\left(4m\ke d\ke +\frac{4}{\aa}\right) < \frac14\left(1+\Lambda_{\rm N}^{-1}\right)^{-1} (C_{\rm tr})^{-2},\\ 
			\sup_{k\in\dZ^2}\left(4m\ke d\ke +\frac{4}{\aa}\right) < \frac14\Lambda_{\rm N} (C_{\rm tr})^{-2},
		\end{aligned}	
	\end{equation}
	for all sufficiently small $\eps > 0$.
	For this choice of $\aa > 0$, we get for all sufficiently small $\eps > 0$ and any $k\in\dZ^2$\allowdisplaybreaks
	\begin{align}
		\notag&\frac12\|\nb v\|^2_{L^2(D\ke)}
		-\left(m\ke +\frac{1}{\aa d\ke}\right)\int_{\p D\ke}|v-\langle v\rangle_{D\ke}|^2\dd\s\\
		\notag&
		 \qquad =\frac14\|\nb (v-\langle v\rangle_{ D\ke})\|^2_{L^2(D\ke)}
		\\\notag
		&\qquad\qquad+\frac14\left[\|\nb (v-\langle v\rangle_{ D\ke})\|^2_{L^2(D\ke)}
			-\left(4m\ke +\frac{4}{\aa d\ke}\right)\int_{\p D\ke}|v-\langle v\rangle_{D\ke}|^2\dd\s\right]\\
		\notag&\qquad \ge 
			\frac{1}{4d\ke^2}\Lambda_{\rm N}\|v-\langle v\rangle_{D\ke}\|^2_{L^2(D\ke)}
			+
			\frac{1}{4d\ke^2}\Lambda_{\rm R}^{-4m\ke d\ke - \frac{4}{\aa}}({\bf D}\ke)\|v-\langle v\rangle_{D\ke}\|^2_{L^2(D\ke)}\\
		\notag&\qquad
		\ge
		\frac{1}{4d\ke^2}
		\left(\Lambda_{\rm N} - \frac{\left(4m\ke d\ke +\frac{4}{\aa}\right)|\p {\bf D}\ke|}{|{\bf D}\ke|}\left(1-\sqrt{\frac{4m\ke d\ke+\frac{4}{\aa}}{\Lambda_{\rm S}({\bf D}\ke)}}\right)^{-2} \right)\|v-\langle v\rangle_{D\ke}\|^2_{L^2(D\ke)}\\
		&\qquad
		\ge
		\frac{1}{4d\ke^2}
		\left(\Lambda_{\rm N} - 4\left(4m\ke d\ke +\frac{4}{\aa}\right)(C_{\rm tr})^2\right)\|v-\langle v\rangle_{D\ke}\|^2_{L^2(D\ke)}\ge 0,\label{eq:nonnegative}		
	\end{align}
	where we used that the gradient of a constant function is zero in the first step, min-max characterisations~\eqref{minmax:N},~\eqref{minmax:R}, assumption~\eqref{assump:N}, and scaling properties~\eqref{muEdelta} in the second step, 
	combined
	the first condition in~\eqref{eq:conditiona}
	with
	the estimates~\eqref{est:LambdaS+}, \eqref{per:est} and the bound \eqref{eq:Robin_weak} in the third and the fourth steps, and finally employed the second condition in~\eqref{eq:conditiona} in the last step.

	Then, we can extend \eqref{eq:lowerbnd} as follows,
	\begin{align}\notag
		\|\sfD_\eps v\|^2_{L^2(\dR^2)} 
		&\geq \frac12\|\nb v\|^2_{L^2(\dR^2)}
		-  4(C_{\rm tr})^2\aa\sum_{k\in\dZ^2} m\ke^2\|v\|^2_{L^2(D\ke)}\\
		\notag&\ge 
		\frac{1}{2}\|\nb v\|^2_{L^2(\dR^2)}
		-\frac{4m_\star^2  (C_{\rm tr})^2\aa C_2}{|{\bf D}\ke|^2}  \sum_{k\in\dZ^2}
		\left\{\frac{\eps^2}{d\ke^2}\|v\|^2_{L^2(R\ke)}\!+\!
		\frac{\eps^4 \ln\left(\frac{\eps}{d\ke}\right) }{d\ke^2}
		\|\nabla v\|^2_{L^2(R\ke)}\right\}\\
		&\ge 
		\frac{1}{2}\|\nb v\|^2_{L^2(\dR^2)}\!
		-\!\frac{36  (m_\star C_{\rm tr})^2\aa C_2}{(\pi\rho^2)^2}  		\left\{\frac{\eps^2}{d\e^2}\|v\|^2_{L^2(\dR^2)}\!+\!
		\frac{\eps^4 \ln\big(\frac{\eps}{d\e}\big) }{d\e^2}
		\|\nabla v\|^2_{L^2(\dR^2)}\right\}\!,	
		\label{eq:lowerbnd+}
	\end{align}
	where in the first step we employ ~\eqref{eq:bterm} and~\eqref{eq:nonnegative}, in the second step
	we apply Lemma~\ref{lemma:6},
	and in the last step we use~\eqref{assump:de},~\eqref{assump:inrad}, and \eqref{YY}.
	Thus, we arrive at the estimate
	\begin{gather}\label{eq:lowerbnd++}
		\|\sfD_\eps v\|^2_{L^2(\dR^2)} 
		\ge
		\left(\frac12 - C_4\frac{\eps^4 \ln\left(\frac{\eps}{d_\eps}\right) }{d_\eps^2}\right)\|\nb v\|^2_{ L^2(\dR^2) }
		- C_4\frac{\eps^2}{d_\eps^2}\|v\|_{L^2(\dR^2)}^2, 
	\end{gather}
	where 
	\begin{gather}\label{C4}
		C_4=\frac{36{(m_\star C_{\rm tr})^2\aa C_2}}{ (\pi\rho^2)^2}.
	\end{gather}
	In the following, we assume that $\eps$ is sufficiently small in order to have (cf.~\eqref{eq:assumption})
	\begin{gather}\label{14a}
		C_4\frac{\eps^4 \ln\left(\frac{\eps} {d_\eps}\right) }{d_\eps^2}\le \frac14. 
	\end{gather}
	Combining \eqref{eq:lowerbnd++} and \eqref{14a} we arrive at
	the final estimate
	\begin{gather}\label{De:final}
		\|\sfD_\eps v\|^2_{L^2(\dR^2)} +\left(C_4\frac{\eps^2}{d_\eps^2}+\frac14\right)\|v\|^2_{L^2(\dR^2)} \ge
		\frac14\|v\|^2_{H^1(\dR^2)}.
	\end{gather}
	
	\subsection*{End of the proof}
	It  follows from \eqref{eq:R3}, \eqref{D:final}, \eqref{De:final} that  for  sufficiently small $\eps$
	one has the estimate
	\[
	\big|\frs_\eps[u,v]\big| \le 2C_3\eps\big(\ln(\tfrac{\eps}{ d_\eps})\big)^\frac12
	\left(\|\sfD u\|^2_{L^2(\dR^2)}\! +\! \|u\|^2_{L^2(\dR^2)}\right)^{\frac12}\left(\|\sfD_\eps v\|^2_{L^2(\dR^2)} \!+\! \left(C_4\frac{\eps^2}{d_\eps^2}+\frac14\right)\|v\|^2_{L^2(\dR^2)}\right)^{\frac12}.
	\]
	Applying Theorem~\ref{thm:abstract}  (with
	$a = 1$, $b = C_4\frac{\eps^2}{d_\eps^2}+\frac14$, $c = 2C_3\eps(\ln(\frac{\eps}{d_\eps}))^{1/2}$) and taking into account \eqref{22} we get  
	\[
	\big\|(\sfD-\ii)^{-1} - (\sfD_\eps - \ii)^{-1}\big\|\le C\frac{\eps^2}{d\e}\left(\ln\left(\frac{\eps}{d_\eps}\right)\right)^\frac12,
	\]
	where
	$
		C=2C_3 \sqrt{2C_4 + \frac{5}{4}}
		$ with $C_3,C_4$ being defined via \eqref{C3},\,\eqref{C4} (the constants $C_1$, $C_2$ standing in the formulae for $C_3$, $C_4$ are given in \eqref{C1}, \eqref{C2}).
	In particular, the family of operators $\sfD_\eps$ converges to $\sfD$ in the norm resolvent sense as $\eps\arr0$.
	
	\section*{Acknowledgments}
	The authors are grateful to V\'it Jakubsk\'y   for fruitful discussions on physical aspects of the investigated model. A.\,K. is  supported by the  Czech Science Foundation (GA\v{C}R) within the project 22-18739S.  V.\,L. is  supported by the  Czech Science Foundation (GA\v{C}R) within the project 21-07129S.   
	
	\newcommand{\etalchar}[1]{$^{#1}$}
		
\end{document}

%% file: commands.tex
\newcommand\nb{\nabla}
\newcommand{\beq}{\begin{equation} \begin{split}}
\newcommand{\eeq}{\end{split} \end{equation}}

\newcommand\Omg{\Omega}

\makeatletter
\def\section{\@startsection{section}{1}\z@{.9\linespacing\@plus\linespacing}%
	{.7\linespacing} {\fontsize{13}{14}\selectfont\bfseries\centering}}
\def\paragraph{\@startsection{paragraph}{4}%
	\z@{0.3em}{-.5em}%
	{$\bullet$ \ \normalfont\itshape}}

\@addtoreset{equation}{section}
\makeatother

\renewcommand\and{\qquad\text{and}\qquad}

\newcommand\sm{\setminus}

\newcommand\one{\mathbbm{1}}

\newcommand{\comm}[1]{}

\def\sfS{\mathsf{S}}
\def\sfS{\mathsf{S}}

\def\bm1{\mathbbm{1}}

\def\s{\sigma}

\def\p{\partial}

\def\one{\mathbbm{1}}
\def\omg{\omega}

\def\Re{{\rm Re}\,}

\def\arr{\rightarrow}

\def\aa{\alpha}

\def\s{\sigma}

\def\ii{{\mathsf{i}}}
\def\p{\partial}

\def\one{\mathbbm{1}}

\def\sfP{\mathsf{P}}
\def\dd{{\,\mathrm{d}}}

\def\omg{\omega}

\def\sfS{\mathsf{S}}

\def\sfP{\mathsf{P}}

\newcounter{counter_a}

\usepackage[latin1]{inputenc}
\usepackage[T1]{fontenc}

\numberwithin{figure}{section}
\numberwithin{equation}{section}
\theoremstyle{plain}
\newtheorem*{thm*}{Theorem}
\newtheorem{thm}{Theorem}[section]

\newtheorem{lem}[thm]{Lemma}

\newtheorem{cor}[thm]{Corollary}

\theoremstyle{remark}
\newtheorem{remark}[thm]{Remark}
\theoremstyle{plain}

\newcommand{\beu}{\begin{equation*}}
\newcommand{\eeu}{\end{equation*}}
\newcommand{\besu}{\begin{equation*}
\begin{aligned}}
\newcommand{\eesu}{\end{aligned}
\end{equation*}}
\newcommand{\bes}{\begin{equation}
\begin{aligned}}
\newcommand{\ees}{\end{aligned}
\end{equation}}

\newcommand\cH{\mathcal H}

\newcommand\frs{\mathfrak s}

\newcommand\eps{\varepsilon}

\newcommand\ov{\overline}
\newcommand\wt{\widetilde}

\newcommand\void[1]{}

\def\ov{\overline}
\def\eps{\varepsilon}

\DeclareMathOperator\tr{tr}

\def\frs{{\mathfrak s}}

      \def\dC{{\mathbb C}}

      \def\dR{{\mathbb R}}

   \def\dZ{{\mathbb Z}}

   \def\sfB{{\mathsf B}}   
\def\sfD{{\mathsf D}}

\def\sfP{{\mathsf P}}   \def\sfQ{{\mathsf Q}}   \def\sfR{{\mathsf R}}
\def\sfS{{\mathsf S}}   \def\sfT{{\mathsf T}}

   \def\cH{{\mathcal H}}   
      
      \def\cO{{\mathcal O}}

\newcommand{\dom}{\mathrm{dom}\,}